\documentclass[11pt,leqno]{article}
\usepackage{color}

\usepackage{graphics,epsfig}
\usepackage{amsthm,amsmath,amsfonts}
\usepackage[ansinew]{inputenc}
\def\R{{\mathbb R}}

\numberwithin{equation}{section}

\newtheorem{theorem}{Theorem}[section]

\newtheorem{cor}[theorem]{Corollary}

\newtheorem{example}[theorem]{Example}

\begin{document}

\title{Optimal Multi-Dimensional Stochastic Harvesting with
Density-dependent Prices}
\author{Luis H. R. Alvarez$^{1}$ \and
Edward Lungu$^{2}$ \and
Bernt \O{}ksendal$^{3,4,5}$}
\date{10 May 2015}

\maketitle

\footnotetext[1]{\  Department of Accounting and Finance,
Turku School of Economics,
FIN-20014 University of Turku,
Finland,
e-mail: luis.alvarez@tse.fi}

\footnotetext[2]{\  Department of Mathematics,
University of Botswana,
B.P. 0022 Gaborone,
Botswana,
e-mail: lungu@mopipi.ub.bw}

\footnotetext[3]{\  Department of Mathematics,
University of Oslo,
Box 1053 Blindern,
N--0316 Oslo, Norway,\newline
e-mail: oksendal@math.uio.no
The research leading to these results has received funding from the
European Research Council under the European Community's Seventh
Framework Program (FP7/2007-2013) / ERC grant agreement no [228087].
}

\footnotetext[4]{\
Norwegian School of Economics,
Helleveien 30, N--5045 Bergen,
Norway}
\footnotetext[5]{This research was carried out with support of CAS - Centre for Advanced Study, at the Norwegian Academy of Science and Letters, within the research program SEFE.}

\begin{abstract}
We prove a verification theorem for a class of singular control problems
which model optimal harvesting with density-dependent prices or optimal
dividend policy with capital-dependent utilities. The result is applied
to solve explicitly some examples of such optimal harvesting/optimal
dividend problems.

In particular, we show that if the unit price \emph{decreases} with population density, then the optimal harvesting policy may not exist in the ordinary sense, but can be expressed as a "chattering policy", i.e. the limit as $\Delta x$ and $\Delta t$ go to $0$ of taking out a sequence of small quantities of size $\Delta x$ within small time periods of size $\Delta t$.
\end{abstract}



\section{Introduction}

The determination of an optimal harvesting policy of a stochastically fluctuating renewable resource is typically subject to at least three key factors affecting either the intertemporal evolution of the resource stock or the incentives of a rational risk neutral harvester. First, the exact size of the harvested stock evolves stochastically
due to environmental or demographical randomness. Second, the interaction between different populations
has obviously a direct effect on the density of the harvested stocks. Third, most harvesting decisions are subject to density dependent costs and prices. The price of the harvested resource is typically decreasing as a function of the
prevailing stock due to the decreasing marginal utility of consumption. The more abundant a resource gets, the
less consumers are prepared to pay from an extra unit of that particular resource and vice versa. In a completely analogous fashion
the costs associated with harvesting depend typically on the abundance of the harvested resource. The scarcer a resource becomes, the higher are the costs associated with
harvesting due to costly search or other similar factors. Our objective in this study is to investigate the optimal harvesting policy of a risk neutral decision maker facing all the three key factors
mentioned above.
\\
The problem of determining an optimal harvesting policy of a risk neutral decision maker can be viewed as a singular stochastic control problem. In an unstructured
one-dimensional setting where the marginal profitability of a marginal unit of the harvested stock is a constant, the existing literature usually delineates circumstances under which
the optimal harvesting policy is to deplete the entire resource stock immediately or to maintain it at all times below a critical threshold at which the expected present value of the cumulative yield
is maximized (\cite{A1,A3,AS,LES1,LES2, LO1}). As intuitively is clear, the optimal policy is altered as soon as the marginal profitability becomes state-dependent (cf. \cite{A2}) or population
interaction (cf. \cite{LO2}) is incorporated into the analysis. In \cite{A2} it is shown within a one-dimensional setting that the state dependence of the instantaneous yield from harvesting
results into the emergence of circumstances under which the policy resulting into the maximal value constitutes a chattering policy which does not belong into the original class of admissible
càdlàg-harvesting policies. On the other hand, in \cite{LO2} it is shown that the presence of interaction between the harvested resource stocks leads to a harvesting strategy where the decision maker
generically harvests only a single resource at a time.
\\

In this paper we combine the approaches developed in \cite{A2} and \cite{LO2} and consider the problem of determining the optimal harvesting policy from a collection
of interacting populations, described by a coupled system of stochastic differential equations, when the price per unit for each population is allowed to depend on the densities of the populations. In Section 2 we give a general verification theorem for such optimal harvesting problems (Theorem 2.1), and in Section 3 we study in detail some examples where the price is a decreasing function of the density and we show, perhaps surprisingly, that in such cases the optimal harvesting strategy may not exist in the ordinary sense, but can be described as a "chattering policy". See Theorem 3.2 and Theorem 3.4.


\section{The main result}

We now describe our model in detail. This presentation follows
\cite{LO2} closely. Consider $n$ populations whose sizes or densities
$X_1(t),\ldots,X_n(t)$ at time $t$ are described by a system of $n$
stochastic differential equations of the form
\begin{align}
dX_i(t) &= b_i(t,X(t))dt+
      \sum_{j=1}^m \sigma_{ij}(t,X(t))dB_j(t);
       0\leq s\leq t\leq T \\
X_i(s) &= x_i\in\R\;; \qquad 1\leq i\leq n\;,
\end{align}
where $B(t)=(B_1(t),\ldots,B_m(t))$; $t\geq0$, $\omega\in\Omega$ is
$m$-dimensional Brownian motion on a filtered probability space $(\Omega, \mathcal{F}, \mathbb{F} := \{ \mathcal{F}_t \}_{t \geq 0} , P) $ and the differentials (i.e. the
corresponding integrals) are interpreted in the It\^o sense. We assume
that $b=(b_1,\ldots,b_n):\R^{1+n}\to\R^n$ and
$\sigma=(\sigma_{ij})_{1\leq i\leq n \hfill \atop 1\leq j\leq m}:\R^{1+n}\to
\R^{n\times m}$ are given continuous functions. We also assume that the
terminal time $T=T(\omega)$ has the form
\begin{equation}
T(\omega)=\inf\big\{ t>s; (t,X(t))\not\in S\big\}
\end{equation}
where $S\subset\R^{1+n}$ is a given set. For simplicity we will assume in
this paper that
\[
S=(0,T)\times U
\]
where $U$ is an open, connected set in $\R^n$.
We may interprete $U$ as the {\em survival set\/} and $T$ is the {\em
time of extinction\/} or simply the {\em closing/terminal\/} time.
\bigskip

We now introduce a {\em harvesting strategy\/} for this family of
populations:

A {\em harvesting strategy\/} $\gamma$ is a stochastic process
$\gamma(t)=\gamma(t,\omega)=(\gamma_1(t,\omega),\ldots,\gamma_n(t,\omega))
\in\R^n$ with the following properties:
\begin{eqnarray}
&&\hbox{For each $t\geq s$ $\gamma(t,\cdot)$ is measurable with respect to
       the $\sigma$-algebra ${\mathcal F}_t$ generated by } \\
&&\qquad \hbox{$\{B(s,\cdot); s\leq t\}$. In other words:
         $\gamma(\cdot)$ is $\mathbb{F}$-{\em adapted}.} \nonumber \\ [1ex]
&& \gamma_i(t,\omega)\ \hbox{is non-decreasing with respect to $t$,
       for a.a. $\omega\in\Omega$ and all $i=1,\ldots,n$} \\ [1ex]
&&t\to\gamma(t,\omega)\ \hbox{is right-continuous,
         for a.a. $\omega$} \\ [1ex]
&&\gamma(s,\omega)=0\ \hbox{for a.a. $\omega$}\;.
\end{eqnarray}
Component number $i$ of $\gamma(t,\omega),\gamma_i(t,\omega)$, represents
{\em the total amount harvested from population number $i$ up to time $t$}.

If we apply a harvesting strategy $\gamma$ to our family
$X(t)=(X_1(t),\ldots,X_n(t))$ of populations the harvested family
$X^{(\gamma)}(t)$ will satisfy the $n$-dimensional stochastic differential
equation
\begin{equation}
\begin{cases}
dX^{(\gamma)}(t)=b(t,X^{(\gamma)}(t))dt+\sigma(t,X^{(\gamma)}(t))
        dB(t)-d\gamma(t)\;;\quad s\leq t\leq T \\
X^{(\gamma)}(s^-)=x=(x_1,\ldots,x_n)\in\R^n\end{cases}
\end{equation}
We let $\Gamma$ denote the set of all harvesting strategies $\gamma$ such
that the corresponding system (2.7) has a unique strong solution
$X^{(\gamma)}(t)$ which does not explode in the time interval $[s,T]$ and
such that $X^{(\gamma)}(t)\in U$ for all $t \in [s,T]$.

Since we do not exclude immediate harvesting at time $t=s$, it is necessary
to distinguish between $X^{(\gamma)}(s)$ and
$X^{(\gamma)}(s^-)$:
Thus $X^{(\gamma)}(s^-)$ is the state right before harvesting
starts at time $t=s$, while
\[
X^{(\gamma)}(s)=X^{(\gamma)}(s^-)-\Delta\gamma
\]
is the state immediately after, if $\gamma$ consists of an immediate
harvest of size $\Delta\gamma$ at $t=s$.

Suppose that {\em the price per unit of population number\/} $i$, when
harvested at time $t$ and when the current size/density of the vector
$X^{(\gamma)}(t)$ of populations is $\xi=(\xi_1,\ldots,\xi_n)\in\R^n$, is
given by
\begin{equation}
\pi_i(t,\xi)\;;\qquad (t,\xi)\in S\;,\quad 1\leq i\leq n\;,
\end{equation}
where the $\pi_i:S\to\R$; $1\leq i\leq n$, are lower bounded continuous
functions. We call such prices {\em density-dependent\/} since they
depend on $\xi$. The total expected discounted utility harvested from
time $s$ to time $T$ is given by
\begin{equation}
J^{(\gamma)}(s,x):=E^{s,x} \Big[
\int\limits_{[s,T)} \pi(t,X^{(\gamma)}(t^-))\cdot d\gamma(t)\Big]
\end{equation}
where $\pi=(\pi_1,\ldots,\pi_n)$, $\pi\cdot d\gamma=\sum\limits_{i=1}^n
\pi_i d\gamma_i$ and $E^{s,x}$ denotes the expectation with respect to the
probability law $Q^{s,x}$ of the time-state process
\begin{equation}
Y^{s,x}(t)=Y^{\gamma,s,x}(t)=(t,X^{(\gamma)}(t))\;;\qquad t\geq s
\end{equation}
assuming that $Y^{s,x}(s^-)=x$.

The {\em optimal harvesting problem\/} is to find the {\em value
function\/} $\Phi(s,x)$ and an {\em optimal harvesting strategy\/}
$\gamma^\ast\in\Gamma$ such that
\begin{equation}
\Phi(s,x):=\sup_{\gamma\in\Gamma}
J^{(\gamma)}(s,x)=J^{(\gamma^\ast)}(s,x)\;.
\end{equation}
This problem differs from the problems considered in \cite{A1},
\cite{A3}, \cite{AS}, \cite{LO1} and \cite{LO2} in that the prices
$\pi_i(t,\xi)$ are allowed to be density-dependent. This allows for more
realistic models. For example, it is usually the case that if a type of
fish, say population number $i$, becomes more scarce, the price per unit
of this fish increases. Conversely, if a type of fish becomes abundant
then the price per unit goes down. Thus in this case the price
$\pi_i(t,\xi)=\pi_i(t,\xi_1,\ldots,\xi_n)$ is a {\em nonincreasing\/}
function of $\xi_i$. One can also have situations where $\pi_i(t,\xi)$
depends on all the other population densities $\xi_1,\ldots,\xi_n$ in a
similar way.

It turns out that if we allow the prices to be density-dependent, a number
of new -- and perhaps surprising -- phenomena occurs. The purpose of this
paper is not to give a complete discussion of the situation, but to consider
some illustrative examples.

\paragraph{Remark}
Note that we can also give the problem (2.12) an economic interpretation:
We can regard $X_i(t)$ as the value at time $t$ of an economic quantity
or asset and we can let $\gamma_i(t)$ represent the total amount paid in
dividends from asset number $i$ up to time $t$. Then $S$ can be
interpreted as the {\em solvency set}, $T$ as the {\em time of
bankruptcy\/} and $\pi_i(t,\xi)$ as the {\em utility rate\/} of dividends
from asset number $i$ at the state $(t,\xi)$. Then (2.12) becomes the
problem of finding the {\em optimal stream of dividends}. This
interpretation is used in \cite{JS} (in the density-independent utility
case). See also
\cite{LO2}.
\bigskip

In the following $H^0$ denotes the interior of a set $H$, $\bar{H}$
denotes its closure.

If $G\subset\R^k$ is an open set we let $C^2(G)$ denote the set of real
valued twice continuously differentiable functions on $G$. We let
$C_0^2(G)$ denote the set of functions in $C^2(G)$ with compact support
in $G$.

If we do not apply any harvesting, then the corresponding time-state
population process $Y(t)=(t,X(t))$, with $X(t)$ given by (2.1)--(2.2), is
an It\^o diffusion whose generator coincides on $C_0^2(\R^{1+n})$ with the
partial differential operator $L$ given by
\begin{equation}
Lg(s,x)=\frac{\partial g}{\partial s}(s,x)+
\sum_{i=1}^n b_i(s,x)\frac{\partial g}{\partial x_i}(s,x)
+\tfrac{1}{2}\sum_{i,j=1}^n
(\sigma\sigma^T)_{ij}(s,x)
\frac{\partial^2g}{\partial s\partial x}
\end{equation}
for all functions $g\in C^2(S)$.

The following result is a generalization to the multi-dimensional case of
Theorem~1 in \cite{A2} and a generalization to density-dependent prices of
Theorem~2.1 in \cite{LO2}. For completeness we give the proof.

\begin{theorem}\label{thm1}
Assume that
\begin{equation}
\pi(t,\xi)\ \hbox{is {\em nonincresing\/} with respect to
$\;\xi_1,\ldots,\xi_n$, for all $t$}\;.
\end{equation}
{\bf a)}\
Suppose $\varphi\geq0$ is a function in $C^2(S)$ satisfying the following
conditions
\begin{itemize}
\item[(i)]
\quad
$\frac{\partial\varphi}{\partial x_i}(t,x)\geq \pi_i(t,x)$\quad for
all
$(t,x)\in S$

\item[(ii)]
\quad
$L\varphi(t,x)\leq0$\quad for all $(t,x)\in S$.
\end{itemize}
\noindent
Then
\begin{equation}
\varphi(s,x)\geq\Phi(s,x)\qquad \hbox{for all $(s,x)\in S$}\;.
\end{equation}
{\bf b)}\
Define the {\em nonintervention region\/} $D$ by
\begin{equation}
D=\Big\{
(t,x)\in S;\frac{\partial\varphi}{\partial x_i}(t,x)>\pi_i(t,x)\
\hbox{for all $\;i=1,\ldots,n$}\Big\}.
\end{equation}
Suppose that, in addition to (i) and (ii) above,
\begin{itemize}
\item[(iii)]
\quad
$L\varphi(t,x)=0$\quad for all $(t,x)\in D$
\end{itemize}

\noindent
and that there exists a harvesting strategy $\hat{\gamma}\in\Gamma$ such
that the following, (iv)--(vii), hold:
\begin{itemize}
\item[(iv)]
\quad
$X^{(\hat{\gamma})}(t)\in\bar{D}$\quad for all $\;t\in[s,T]$

\item[(v)]
\quad
$\big(\frac{\partial\varphi}{\partial x_i}(t,X^{(\hat{\gamma})}(t))
-\pi_i(t,X^{(\hat{\gamma})}(t))\big)\cdot
d\hat{\gamma}_i^{(c)}(t)=0$;\quad $1\leq i\leq n$ (i.e.
$\hat{\gamma}_i^{(c)}$ increases only when $\frac{\partial\varphi}{\partial
x_i}=\pi_i$)

and
\item[(vi)]
\quad
$\varphi(t_k,X^{(\hat{\gamma})}(t_k))-\varphi(t_k,
X^{(\hat{\gamma})}(t_k^-))
=-\pi_i(t_k,X^{(\hat{\gamma})}(t_k^-))\cdot \Delta\hat{\gamma}(t_k)$

at all jumping times $t_k\in[s,T)$ of $\hat{\gamma}(t)$, where
\[
\Delta\hat{\gamma}(t_k)=\hat{\gamma}(t_k)-\hat{\gamma}(t_k^-)
\]
and
\item[(vii)]
\quad
$E^{s,x}\big[ \varphi(T_R,X^{(\hat{\gamma})}(T_R))\big]\to0$\quad
as $R\to\infty$

where
\[
T_R=T\wedge R\wedge \inf \big\{ t>s; |X^{(\hat{\gamma})}(t)|\geq
R\big\}\;;\qquad R>0\;.
\]
Then
\begin{equation}
\varphi(s,x)=\Phi(s,x)\qquad \hbox{for all $(s,x)\in S$}
\end{equation}
and
\[
\gamma^\ast:=\hat{\gamma}\quad \hbox{is an optimal harvesting strategy}.
\]
\end{itemize}
\end{theorem}

\begin{proof}
{\bf a)}\
Choose $\gamma\in\Gamma$ and $(s,x)\in S$. Then by It\^o's formula for
semimartingales (the Dol\'eans-Dade-Meyer formula) \cite[Th.~II.7.33]{P}
we have
\begin{eqnarray}
\lefteqn{
E^{s,x}[ \varphi(T_R,X^{(\gamma)}(T_R^-))]
        =E^{s,x}[\varphi(s,X^{(\gamma)}(s))]} \nonumber \\
&&+E^{s,x} \Big[
     \int\limits_s^{T_R} \frac{\partial\varphi}{\partial t}
     (t,X^{(\gamma)}(t))dt+\int\limits_{(s,T_R)} \sum_{i=1}^n
      \frac{\partial\varphi}{\partial x_i}(t,X^{(\gamma)}(t^-))
        dX_i^{(\gamma)}(t) \nonumber \\
&&+\sum_{i,j=1}^n\, \int\limits_s^{T_R}\tfrac{1}{2}(\sigma\sigma^T)_{ij}
      (t,X^{(\gamma)}(t))\frac{\partial^2\varphi}{\partial x_i\partial x_j}
       (t,X^{(\gamma)}(t))dt \nonumber \\
&&+\sum_{s<t_k<T_R} \Big\{
     \varphi(t_k,X^{(\gamma)}(t_k))
     -\varphi(t_k,X^{(\gamma)}(t_k^-))
      -\sum_{i=1}^n
      \frac{\partial\varphi}{\partial x_i}(t_k,X^{(\gamma)}(t_k^-))
     \Delta X_i^{(\gamma)}(t_k)\Big\} \Big],
\end{eqnarray}
where the sum is taken over all jumping times $t_k\in(s,T_R)$ of
$\gamma(t)$ and
\[
\Delta X_i^{(\gamma)}(t_k)=X_i^{(\gamma)}(t_k)
-X_i^{(\gamma)}(t_k^-)\;.
\]
Let $\gamma^{(c)}(t)$ denote the continuous part of $\gamma(t)$, i.e.
\[
\gamma^{(c)}(t)=\gamma(t)-\sum_{s\leq t_k\leq t}
\Delta \gamma(t_k)\;.
\]
Then, since $\Delta X_i^{(\gamma)}(t_k)=-\Delta \gamma_i(t_k)$ we see
that (2.18) can be written
\begin{eqnarray}
\lefteqn{E^{s,x}[\varphi(T_R,X^{(\gamma)}(T_R^-))]
       =\varphi(s,x)} \nonumber \\
&&+E^{s,x} \Big[ \int\limits_s^{T_R} \Big\{
      \frac{\partial\varphi}{\partial t} + \sum_{i=1}^n b_i
      \frac{\partial\varphi}{\partial x_i} +\tfrac{1}{2}\sum_{i,j=1}^n
        (\sigma\sigma^T)_{ij}\, \frac{\partial^2\varphi}{\partial
      x_i\partial x_j}\Big\} (t,X^{(\gamma)}(t))dt\Big] \nonumber \\
&&-E^{s,x} \Big[
     \int\limits_s^{T_R} \sum_{i=1}^n
     \frac{\partial\varphi}{\partial x_i}(t,X^{(\gamma)}(t))
      d\gamma_i^{(c)}(t)\Big] +E^{s,x} \Big[ \sum_{s\leq t_k<T_R}
       \Delta\varphi(t_k,X^{(\gamma)}(t_k))\Big]
\end{eqnarray}
where
\[
\Delta\varphi(t_k,X^{(\gamma)}(t_k))=
\varphi(t_k,X^{(\gamma)}(t_k))-\varphi(t_k,X^{(\gamma)}(t_k^-))\;.
\]
Therefore
\begin{eqnarray}
\lefteqn{\hspace*{-2em}
\varphi(s,x)=E^{s,x}[\varphi(T_R,X^{(\gamma)}(T_R^-))]
       -E^{s,x} \Big[ \int\limits_s^{T_R}
      L\varphi(t,X^{(\gamma)}(t))dt \Big]} \nonumber \\
&&+E^{s,x} \Big[ \int\limits_s^{T_R} \sum_{i=1}^n
       \frac{\partial\varphi}{\partial x_i}
      (t,X^{(\gamma)}(t))d\gamma_i^{(c)}(t)\Big] \nonumber \\
&&-E^{s,x} \Big[ \sum_{s\leq t_k<T_R}
\Delta\varphi(t_k,X^{(\gamma)}(t_k))\Big].
\end{eqnarray}
Let $y=y(r)$; $0\leq r\leq 1$ be a smooth curve in $U$ from
$X^{(\gamma)}(t_k)$ to
$X^{(\gamma)}(t_k^-)=X^{(\gamma)}(t_k)+\Delta\gamma(t_k)$. Then
\begin{equation}
-\Delta\varphi(t_k,X^{(\gamma)}(t_k))=
\int\limits_o^1 \nabla\varphi(t_k,y(r))dy(r)\;.
\end{equation}
We may assume that
\[
dy_i(r)\geq0\qquad \hbox{for all $\;i,r$}\;.
\]
Now suppose that (i) and (ii) hold. Then by (2.20) and (2.21) we have
\begin{align}
\varphi(s,x) \geq E^{s,x} & \Big[
      \int\limits_s^{T_R} \sum_{i=1}^n
      \pi_i(t,X^{(\gamma)}(t))d\gamma_i^{(c)}(t)\Big] \nonumber \\
&+E^{s,x} \Big[ \sum_{s\leq t_k<T_R} \Big(
\int\limits_0^1\, \sum_{i=1}^n \pi_i(t_k,y(r))dy_i(r)\Big)\Big]
\end{align}
Since we have assumed that $\pi_i(t,\xi)$ is {\em nonincreasing\/}
with respect to $\xi_1,\ldots,\xi_n$ we have
\[
\pi_i(t_k,X^{(\gamma)}(t_k^-))\leq \pi_i(t_k,y(r))\leq
\pi_i(t_k,X^{(\gamma)}(t_k))
\]
for all $i,k$ and $r\in[0,1]$. Hence
\begin{equation}
\int\limits_0^1 \pi_i(t_k,y(r))dy_i(r)\geq
\pi_i(t_k,X^{(\gamma)}(t_k^-))
\cdot \Delta\gamma_i(t_k)\;.
\end{equation}
Combined with (2.22) this gives
\begin{align}
\varphi(s,x) &\geq E^{s,x} \Big[
    \int\limits_0^{T_R} \pi(t,X^{(\gamma)}(t))d\gamma^{(c)}(t)
     +\sum_{s\leq t_k<T} \pi(t_k,X^{(\gamma)}(t_k^-))\cdot
     \Delta\gamma(t_k)\Big] \nonumber \\
&=E^{s,x} \Big[ \int\limits_{[s,T_R)}
\pi(t,X^{(\gamma)}(t^-))d\gamma(t)\Big].
\end{align}
Letting $R\to\infty$ we obtain $\varphi(s,x)\geq J^{(\gamma)}(s,x)$.
Since $\gamma\in\Gamma$ was arbitrary we conclude that (2.15) holds.
Hence a) is proved.
\bigskip

\noindent
{\bf b)}\
Next, suppose that (iii)--(vii) also hold. Then if we apply the
argument above to $\gamma=\hat{\gamma}$ we get in (2.20) the
following:
\begin{align*}
\varphi(s,x) &= E^{s,x}[\varphi(T_R,X^{(\hat{\gamma})}(T_R^-))] \\
&\qquad +E^{s,x} \Big[
     \int\limits_0^{T_R} \pi(t,X^{(\hat{\gamma})}(t))\cdot
      d\hat{\gamma}^{(c)}(t)
     +\sum_{s\leq t_k<T_R} \pi(t_k,X^{(\hat{\gamma})}(t_k^-))\cdot
      \Delta\hat{\gamma}(t_k)\Big] \\
&= E^{s,x}[\varphi(T_R,X^{(\hat{\gamma})}(T_R^-))]
      +E^{s,x} \Big[ \int\limits_{[s,T_R)}
      \pi(t,X^{(\hat{\gamma})}(t))\cdot d\hat{\gamma}(t)\Big] \\
&\longrightarrow J^{(\hat{\gamma})}(s,x)\qquad \hbox{as $R\to\infty$}\;.
\end{align*}
Hence $\varphi(s,x)=J^{(\hat{\gamma})}(s,x)\leq\Phi(s,x)$. Combining
this with (2.14) from a) we get the conclusion (2.16) of part b).
This completes the proof of Theorem~2.1.
\end{proof}

If we specialize to the 1-dimensional case with just one population
$X^{(\gamma)}(t)$ given by
\begin{equation}
\begin{cases}
dX^{(\gamma)}(t)=b(t,X^{(\gamma)}(t))dt
       +\sigma(t,X^{(\gamma)}(t))dB(t)-d\gamma(t)\;;\quad t\geq s \\
X^{(\gamma)}(s^-)=x\in\R \phantom{\int^i} \end{cases}
\end{equation}
then Theorem~2.1a) gets the form (see also \cite[Lemma~1]{A2})

\begin{cor}
Assume that
\begin{align}
& \xi\to\pi(t,\xi);\ \xi\in\R\quad
       \hbox{is nonincreasing for all $\;t\in[0,T]$} \\
& \varphi(t,x)\geq0\phantom{\int^i}\quad
     \hbox{is a function in $\;C^2(S)$ such that} \\
& \frac{\partial\varphi}{\partial x}(t,x)\geq \pi(t,x)\qquad
       \hbox{for all $\;(t,x)\in S$}
\end{align}
and
\begin{equation}
L\varphi(t,x)\leq0\qquad \hbox{for all $(\;t,x)\in S$}\;.
\end{equation}
Then
\begin{equation}
\varphi(s,x)\geq\Phi(s,x)\qquad \hbox{for all $\;(s,x)\in S$}\;.
\end{equation}
   \end{cor}

\section{Examples}

In this section we apply Theorem~2.1 or Corollary~2.2 to some special
cases.

\begin{example}
\rm

Suppose $X^{(\gamma)}(t)= (X_1^{(\gamma)}(t),X_2^{(\gamma)}(t))$ is given by
\begin{equation}
\begin{cases}
dX_i^{(\gamma)}(t)=\mu_i\,dt+\sigma_i\,dB_i(t)-d\gamma_i(t)\;; \quad t\geq s \\
X_i^{(\gamma)}(s)=x_i>0 \phantom{\int^i} \end{cases}
\end{equation}
where $\mu_i>0$ and $\sigma_i\not=0$ are constants; $i = 1,2,$ and $\gamma=(\gamma_1,\gamma_2)$.

We want to maximize the total discounted value of the harvest, given by
\begin{equation}
J^{(\gamma)}(s,x)=E^{s,x} \Big[
\int\limits_{[s,T)} e^{-\rho t} \{ g_1(X_1^{(\gamma)}(t^-))d\gamma_1(t)+g_2(X_2^{(\gamma)}(t^-))d\gamma_2(t)\Big]
\end{equation}
where $g_i:\R\to\R$ are given nonincreasing functions (the
density-dependent prices) and
\begin{equation}
T=\inf\big\{ t>s; \min(X_1^{(\gamma)}(t),X_2^{(\gamma)}(t))\leq 0\big\}
\end{equation}
is the time of extinction, i.e. $S=\{(t,x);x_i>0; i=1,2\}$. The corresponding 1-dimensional case with $g$ {\em
constant\/} was solved in \cite{JS}. Then it is optimal to do nothing if
the population is below a certain treshold $x^\ast>0$ and then harvest
according to {\em local time\/} of the downward reflected process
$\bar{X}(t)$ at $\bar{X}(t)=x^\ast$.

Now consider the case when
\begin{equation}
g_i(x)=\theta_i x^{-1/2},\qquad \hbox{i.e.}\quad
\pi_i(t,x)=e^{-\rho t}\theta_ix^{-1/2};\quad x>0\;,
\end{equation}
where $\theta_i >0$ are given constants; $i=1,2$.
Then the prices increase as the population sizes $x_i$ go to 0, so (2.24)
holds. Suppose we apply the ``{\em take the money and run\/}''-strategy
$\;\overset{\circ}{\!\gamma}$. This strategy empties the
whole population immediately. It can be described by
\begin{equation}
\overset{\circ}{\!\gamma}\,(s)=(X_1(s^-),X_2(s^-))=(x_1,x_2)\;.
\end{equation}
Such a strategy gives the harvest value
\begin{equation}
J^{(\overset{\circ}{\!\gamma})}(s,x)=e^{-\rho s} (\theta_1x_1^{-1/2}x_1+\theta_2x_2^{-1/2}x_2)
=e^{-\rho s}(\theta_1\sqrt{x_1}+\theta_2\sqrt{x_2})\;;\qquad x_i>0\;.
\end{equation}
However, it is unlikely that this is the best strategy because it does
not take into account that the prices increase as the population sizes
go down. So for the two populations $X_i(t); i=1,2,$ we try the following ``chattering policy'', denoted by
$\widetilde{\gamma}_i=\widetilde{\gamma}_i^{(m,\eta)}$, where $m$ is a fixed
natural number and $\eta>0$:

At the times
\begin{equation}
t_k=\Big(s+\frac{k}{m}\eta\Big)\wedge T\;; \qquad k=1,2,\ldots,m
\end{equation}
we harvest an amount $\Delta\widetilde{\gamma}_i(t_k)$ which is the
fraction $\frac{1}{m}$ of the current population. This gives the expected
harvest value
\begin{equation}
J^{(\tilde{\gamma}(m,\eta))}(s,x)=
E^{s,x} \Big[ \sum_{k=1}^m e^{-\rho t_k}
      [\theta_1 \big( X_1^{(\tilde{\gamma})}(t_k^-))^+ \big)^{-1/2}\Delta\widetilde{\gamma}_1(t_k)+\theta_2 \big(X_2^{(\tilde{\gamma})}(t_k^-))^+ \big)^{-1/2} \Delta\widetilde{\gamma}_2(t_k)] \Big]
\;,
\end{equation}
where we have used the notation
\[
x_i^+=\max(x_i,0)\;;\qquad x_i\in\R\;.
\]
Now let $\eta\to0,m\to \infty$. Then all the $t_k$'s converge to $s$ and we get
\begin{align}
&J^{(\tilde{\gamma}(m,0))}(s,x):= \lim_{\eta\to 0,m\to \infty}
      J^{(\tilde{\gamma}(m,\eta))}(s,x)\nonumber\\
&= \lim_{m\to \infty}e^{-\rho s} \big[ \sum_{k=1}^m
      \theta_1\Big( x_1-\frac{k}{m}x_1\Big)^{-1/2}\frac{1}{m}x_1+\sum_{k=1}^m
      \theta_2\Big( x_2-\frac{k}{m}x_2\Big)^{-1/2}\frac{1}{m}x_2\big] \nonumber \\
&= e^{-\rho s}\big[\theta_1 x_1^{\frac{1}{2}}\int_0^1 (1-y)^{-\frac{1}{2}}dy+\theta_2 x_2^{\frac{1}{2}}\int_0^1 (1-y)^{-\frac{1}{2}}dy\big]\nonumber\\
&=2 e^{-\rho s}\big[\theta_1 \sqrt{x_1}+\theta_2 \sqrt{x_2}\big]
\;.
\end{align}

We conclude that
\begin{equation}\label{eq3.10}
\sup_\gamma J^{(\gamma)}(s,x)\geq 2 e^{-\rho s}\big[\theta_1 \sqrt{x_1}+\theta_2 \sqrt{x_2}\big]\;.
\end{equation}
We call this policy of applying
$\widetilde{\gamma}^{(m,\eta)}$ in the limit as $\eta \to 0$ and
$m \to \infty$ the {\em policy of immediate chattering down to 0}. (This
limit does not exist as a strategy in $\Gamma$.)
  From \eqref{eq3.10} we conclude that
\begin{equation}
\Phi(s,x)\geq 2 e^{-\rho s}\big[\theta_1 \sqrt{x_1}+\theta_2 \sqrt{x_2}\big]\;.
\end{equation}
On the other hand, let us check if the function
\begin{equation}
\varphi(s,x):=2 e^{-\rho s}\big[\theta_1 \sqrt{x_1}+\theta_2 \sqrt{x_2}\big]
\end{equation}
satisfies the conditions of Theorem~2.1: Condition
(2.14) holds trivially, and (i) of Part a) holds, since
\[
\frac{\partial\varphi}{\partial x_i}(s,x)=e^{-\rho s}
\theta_1 x_1^{-1/2}=\pi_i(s,x)\;.\qquad \hbox{ }\;
\]
Now
\[
L=\frac{\partial}{\partial s}+\mu_1
\frac{\partial}{\partial
x_1}+\mu_2
\frac{\partial}{\partial
x_2}+\tfrac{1}{2}\sigma_1^2\frac{\partial^2}{\partial x_1^2}+\tfrac{1}{2}\sigma_2^2\frac{\partial^2}{\partial x_2^2},
\]
and therefore
\begin{align*}
L\varphi(s,x) &= 2e^{-\rho s} \big[ -\rho (\theta_1x_1^{1/2}+\theta_2 x_2^{1/2})
+\mu_1\theta_1\tfrac{1}{2}x_1^{-1/2}+\mu_2 \theta_2 \tfrac{1}{2}x_2^{-1/2}+\tfrac{1}{2}
\sigma_1^2\tfrac{1}{2}(-\tfrac{1}{2})\theta_1x_1^{-3/2}+\tfrac{1}{2}
\sigma_2^2\tfrac{1}{2}(-\tfrac{1}{2})x_2^{-3/2}\big] \\
&= -2\rho e^{-\rho s} \Big[ \theta_1 x_1^{-3/2} ( x_1^2-\frac{\mu_1}{2\rho}x_1
+\frac{\sigma_1^2}{8\rho})+\theta_2x_2^{-3/2} ( x_2^2-\frac{\mu_2}{2\rho}x_2
+\frac{\sigma_2^2}{8\rho})\Big].
\end{align*}
So (ii) of Theorem 2.1 a) holds if $\mu_i^2\leq 2\rho\sigma_i^2$ for $i=1,2$. By Theorem~2.1 we
conclude that $\varphi=\Phi$ in this case.
\end{example}

We have proved part a) of the following result:

\begin{theorem}
Let $X^{(\gamma)}(t)$ and $T$ be given by (3.1) and (3.3), respectively.
\bigskip

\noindent
{\bf a)}\
Assume that
\begin{equation}
\mu_i^2\leq 2\rho\sigma_i^2\;, \quad i=1,2.
\end{equation}
Then
\begin{align}
\Phi(s,x)&:=\sup_{\gamma\in\Gamma} E^{s,x} \Big[
\int\limits_{[s,T)} e^{-\rho t} \{ \theta_1 X_1^{(\gamma)}(t^-)^{-1/2}
d\gamma_1(t)+\theta_2 X_2^{(\gamma)}(t^-)^{-1/2}
d\gamma_2(t)\}\Big]\nonumber\\
&=2 e^{-\rho s}\big[\theta_1 \sqrt{x_1}+\theta_2 \sqrt{x_2}\big]\;.
\end{align}
This value is achieved in the limit if we apply the strategy
$\widetilde{\gamma}^{(m,\eta)}$ above with $\eta\to0$ and $m\to\infty$,
i.e. by applying the policy of immediate chattering down to 0.
\bigskip\\
\noindent
{\bf b)}\

Assume that
\begin{equation}
\mu_i^2>2\rho\sigma_i^2 ; \quad i=1,2.
\end{equation}
Then the value function has the form
\begin{equation}\label{eq3.16}
\Phi(s,x)= \begin{cases}
e^{-\rho s}\Big[C_1(e^{\lambda_1^{(1)}x_1}-e^{\lambda_2^{(1)}x_1})+C_2(e^{\lambda_1^{(2)}x_2}-e^{\lambda_2^{(2)}x_2})\Big]\;; & x_1\leq x_1^\ast ; x_2 \leq x_2^*\\
e^{-\rho s}\big[ 2\theta_1\sqrt{x_1}-2\theta_1\sqrt{x_1^\ast}+C_2(e^{\lambda_1^{(2)}x_2}-e^{\lambda_2^{(2)}x_2})+A_1\Big]\;; & x_1 > x_1^*, x_2 \leq x_2^*\\
e^{-\rho s}\Big[ C_1(e^{\lambda_1^{(1)}x_1}-e^{\lambda_2^{(1)}x_1})+2\theta_2\sqrt{x_2}-2\theta_2\sqrt{x_2^\ast}+A_2\Big]\;; & x_1 \leq x_1^*; x_2 > x_2^*\\
e^{-\rho s}\Big[ 2\theta_1\sqrt{x_1}-2\theta_1\sqrt{x_1^\ast}+2\theta_2\sqrt{x_2}-2\theta_2\sqrt{x_2^\ast}+A_1+A_2\Big]\;; & x_1 > x_1^*; x_2 > x_2^*
\end{cases}
\end{equation}
for constants $C_i>0$, $A_i>0$ and $x_i^\ast>0; i=1,2$ satisfying the following system of 6 equations
(see Remark below):
\begin{align}\label{eq3.17}
&C_i(e^{\lambda_1^{(i)}x_i^\ast}-e^{\lambda_2^{(i)} x_i^\ast})=A _i\phantom{\sum}; \quad i=1,2\nonumber\\
&C_i(\lambda_1^{(i)} e^{\lambda_1^{(i)}x_i^\ast}-
       \lambda_2^{(i)}e^{\lambda_2^{(i)} x_i^\ast})=(x_i^\ast)^{-1/2} \phantom{\sum}; \quad i=1,2\nonumber\\
&C_i( (\lambda_1^{(i)})^2 e^{\lambda_1^{(i)}x_i^\ast}-(\lambda_2^{(i)})^2 e^{\lambda_2^{(i)}x_i^\ast})
       =-\tfrac{1}{2}(x_i^\ast)^{-3/2}; \quad i=1,2,
\end{align}

where
\begin{equation}\label{eq3.18}
\lambda_1^{(i)}=\sigma_i^{-2}\big[
-\mu_i+\sqrt{\mu_i^2+2\rho\sigma_i^2}\,\big]>0\;,\quad
\lambda_2^{(i)}=\sigma_i^{-2}\big[ -\mu_i-\sqrt{\mu_i^2+2\rho\sigma_i^2}\,\big]<0\;.
\end{equation}
The corresponding optimal policy is the following, for $i=1,2$:
\begin{align}
&\hbox{If $\;x_i>x_i^\ast$ it is optimal to apply immediate chattering from
       $x_i$ down to $x_i^\ast$.}\phantom{\int} \\
&\hbox{if $\;0<x_i \leq x_i^\ast$ it is optimal to apply the harvesting equal
       to the {\em local time of}} \\
&\hbox{{\em the downward reflected process\/}
$\bar{X}_i(t)$ at $x_i^\ast$.} \nonumber
\end{align}

\noindent
{\bf c)}\
Assume that
\begin{equation}
\mu_1^2>2\rho\sigma_1^2 \text{  and  } \mu_2^2\leq 2\rho\sigma_2^2.
\end{equation}
Then the value function has the form
\begin{equation}
\Phi(s,x)= \begin{cases}
e^{-\rho s} \Big[ C_1(e^{\lambda_1x_1}-e^{\lambda_2x_1}) + 2\theta_2 \sqrt{x_2}\Big]\;; & 0\leq x_1<x_1^\ast \\
e^{-\rho s} \Big[ 2\sqrt{x_1}-2\sqrt{x_1^\ast}+A_1+ 2\theta_2 \sqrt{x_2} \Big]\;; & x_1^\ast\leq x_1
\end{cases}
\end{equation}
for constants $C_1>0$, $A_1>0$ and $x_1^\ast>0$ specified by the 3
equations
\begin{align}
&C_1(e^{\lambda_1x_1^\ast}-e^{\lambda_2 x_1^\ast})=A _1\phantom{\sum}\\
&C_1(\lambda_1e^{\lambda_1x_1^\ast}-
       \lambda_2e^{\lambda_2x_1^\ast})=(x_1^\ast)^{-1/2} \phantom{\sum} \\
&C_1(\lambda_1^2 e^{\lambda_1x_1^\ast}-\lambda_2^2e^{\lambda_2x_1^\ast})
       =-\tfrac{1}{2}(x_1^\ast)^{-3/2},
\end{align}

where
\begin{equation}
\lambda_1=\sigma_1^{-2}\big[
-\mu_1+\sqrt{\mu_1^2+2\rho\sigma_1^2}\,\big]>0\;,\quad
\lambda_2=\sigma_1^{-2}\big[ -\mu_1-\sqrt{\mu_1^2+2\rho\sigma_1^2}\,\big]<0\;.
\end{equation}
The corresponding optimal policy $\gamma^*= (\gamma_1^{*},\gamma_2^*)$ is described as follows: \\
\begin{align}
&\hbox{If $\;x_1>x_1^\ast$ the optimal $\gamma_1^*$ is to apply immediate chattering from
       $x_1$ down to $x_1^\ast$.}\phantom{\int} \\
&\hbox{if $\;0<x_1\leq x_1^\ast$ the optimal $\gamma_1^*$ is to apply the harvesting equal
       to the {\em local time of}} \\
&\hbox{{\em the downward reflected process\/}
$\bar{X}_1(t)$ at $x_1^\ast$.} \nonumber
\end{align}

The optimal policy $\gamma_2^{*}$ is to apply immediate chattering from $x_2$ down to 0.
\end{theorem}

\noindent
\begin{proof}
{\em b)}.\quad
First note that if we apply the policy of immediate chattering from $x_i$
down to $x_i^\ast$, where $0<x_i^\ast<x_i$, then the value of the harvested
quantity is
\begin{equation}
e^{-\rho s}\theta_i \int\limits_0^{x_i-x_i^\ast}(x_1-y)^{-1/2}dy=e^{-\rho s}\theta_i
       \int\limits_{x_i^\ast}^{x_i} u^{-1/2}du=
2e^{-\rho s}\theta_i\big(\sqrt{x_i}-\sqrt{x_i^\ast}\,\big)\;.
\end{equation}
This follows by the argument (3.7)--(3.12) above.

To verify \eqref{eq3.16}--\eqref{eq3.18}, first note that $\lambda_1^{(i)},\lambda_2^{(i)}$ are the roots of
the quadratic equation
\begin{equation}
-\rho+\mu_i\lambda+\tfrac{1}{2}\sigma_i^2\lambda^2=0\;.
\end{equation}
Hence, with $\varphi(s,x)$ defined to be the right hand side of (3.16) we
have
\begin{align}\label{eq3.28}
&L\varphi(s,x)=0 \qquad \hbox{for $\;x_1<x_1^\ast, x_2 < x_2^*$} \\
&L\varphi(s,x) \leq 0 \qquad \hbox{for $\;x_1>x_1^\ast \text{ or } x_2 > x_2^*$} \nonumber\\
\intertext{and}
&\varphi(s,0)=0\;.
\end{align}
Note that equations \eqref{eq3.17} imply that $\varphi$ is $C^2$ at $x_1=x_1^\ast$ and at $x_2 = x_2^*$.

We conclude that with this choice of
$C_i,A_i,x_i^\ast; i=1,2$ the function $\varphi(s,x)$ becomes a $C^2$ function and
the nonintervention region $D$ given by (2.16) is seen to be
\[
D=\{(s,x)=(s,x_1,x_2);0<x_1<x_1^*, 0 < x_2 < x_2^*\}\;.
\]
Thus we obtain that $\varphi$ satisfies conditions (i), (ii) of Theorem~2.1
and hence
\begin{equation}\label{eq3.30}
\varphi(s,x)\geq\Phi(s,x)\qquad \hbox{for all $\;s,x$}\;.
\end{equation}
Also, by (3.31) we know that (iii) holds.

Moreover, if $x_i\leq x_i^\ast$ it is well-known that the local time
$\hat{\gamma_i}$ at $x_i^\ast$ of the downward reflected process $\bar{X}_i(t)$
at $x_i^\ast$ satisfies (iv)--(vi). (See e.g. \cite{LO1} for more details.)
And (vii) follows from (3.16). By Theorem~2.1 b) we conclude that if
$x_i\leq x_i^\ast$ then $\gamma_i^\ast:=\hat{\gamma}_i$ is optimal for $i=1,2$ and
$\varphi(s,x)=\Phi(s,x)$. Finally, as seen above, if $x_i>x_i^\ast$ then immediate chattering from $x_i$ down to $x_i^\ast$ gives the value $2e^{-\rho s}\theta_i \big( \sqrt{x_i}-\sqrt{x_i^\ast}\,\big)+\Phi(s,x^\ast)$.
Hence
\[
\Phi(s,x)\geq 2e^{-\rho s} \theta_i \big(
\sqrt{x_i}-\sqrt{x_i^\ast}\,\big)+\Phi(s,x^\ast)\qquad \hbox{for
$\;x_i > x_i^\ast$}; i=1,2\;.
\]
Combined with \eqref{eq3.30} this shows that
\[
\varphi(s,x)=\Phi(s,x)\qquad \hbox{for all $\;s,x$}
\]
and the proof of b) is complete.\\
The proof of the mixed case c) is left to the reader.
\end{proof}
\noindent \emph{Remark}\\
Dividing the second equation of \eqref{eq3.17} by the third, we get the equation
\begin{equation}\label{eq3.31}
\frac{\lambda_1^{(i)} e^{\lambda_1^{(i)}x_i^\ast}-\lambda_2^{(i)} e^{\lambda_2^{(i)}x_i^\ast}}
{(\lambda_1^{(i)})^2 e^{\lambda_1^{(i)}x_i^\ast}-(\lambda_2^{(i)})^2  e^{\lambda_2^{(i)}x_i^\ast}}
=-2x_i^\ast\;.
\end{equation}
Since the left hand side of \eqref{eq3.31} goes to
$(\lambda_1^{(i)}+\lambda_2^{(i)})^{-1}< 0$ as $x_i^\ast \to 0^{+}$, and goes to
$(\lambda_1^{(i)})^{-1}>0$ as $x_i^\ast \to \infty$, we see by the intermediate value
theorem that there exist $x_i^\ast>0; i=1,2$ satisfying this equation. With these values of $x_i^*; i=1,2$ we see that there exists a unique solution $C_i, A_i; i=1,2$ of the system \eqref{eq3.17}.

\begin{example}
\rm
The Brownian motion example is perhaps not so good as a model of a biological stock,
since Brownian motion is a poor model for population growth. Instead,
let us consider a standard population growth model
(in the sense that it
can be generated from a classic birth-death-process), like the logistic
diffusion considered in
\cite{AS}. That is, let us consider the problem
\begin{equation}
V(0,x)=V(x) =
\sup_{\gamma\in\Gamma}E^{x}\int\limits_{[0,T)}
e^{-\rho t}X^{-1/2}(t^-)d\gamma(t)
\end{equation}
subject to
\begin{equation}
dX(t) = \mu X(t) (1 - K^{-1} X(t))dt + \sigma X(t)dB(t) - d\gamma(t),
\qquad X(0^-) = x>0\;,
\end{equation}
where $\mu > 0$, $K^{-1} > 0$, and $\sigma > 0$ are known constants,
$B(t)$ denotes a Brownian motion in $\R$, and $T = \inf\{t\geq 0:
X(t) \leq 0\}$ denotes the extinction time. We define the mapping
$H:\mathbb{R}_{+}\mapsto\mathbb{R}_{+}$ as
\begin{equation}
H(x) =
\int\limits_{0}^{x}y^{-1/2}dy = 2\sqrt{x}\;.
\end{equation}
The generator $A$ of $X(t)$ is given by
\[
A = \tfrac{1}{2}\sigma^{2}x^{2}\frac{d^{2}}{dx^{2}} + \mu x (1
-  K^{-1} x)\frac{d}{dx}
\]
and we find that
\begin{equation}
G(x): = ((A - \rho)H)(x) = \sqrt{x}\left[\mu - 2\rho - \sigma^{2}/4 -
\mu K^{-1} x\right]\;.
\end{equation}
Thus, if $\mu \leq 2\rho + \sigma^{2}/4$ then
by the same argument as in Example~3.2 we see that the optimal policy is
{\em immediate chattering down to 0}. We then have $T=0$, and the value
reads as
\begin{equation}
V(x) = 2\sqrt{x}\;.
\end{equation}
However, if $\mu > 2\rho + \sigma^{2}/4$, then we see that the mapping
$G(x)$ satisfies the conditions of Theorem 2 in \cite{A2} and, therefore
we find that there is a unique threshold $x^{*}$ satisfying the condition
\begin{equation}
x^{*}\psi''(x^{*}) + \tfrac{1}{2}\psi'(x^{*}) = 0\;,
\end{equation}
where $\psi(x)$ denotes the increasing fundamental solution of the
ordinary differential equation
$((A - \rho)u)(x) = 0$, that is, $\psi(x) = x^{\theta}M(\theta,
2\theta +
\frac{2\mu}{\sigma^2}, \frac{2\mu K^{-1}}{\sigma^2}x)$, where $\theta =
\frac{1}{2} -
\frac{\mu}{\sigma^2} + \sqrt{(\frac{1}{2} - \frac{\mu}{\sigma^2})^{2} +
\frac{2r}{\sigma^2}}\;$, and
$M$ denotes the confluent hypergeometric function. In this case, the
value  reads as
\begin{equation}
V(x) = \begin{cases}
2(\sqrt{x} - \sqrt{x^{*}}) + \sqrt{x^{*}}(\mu(1-K^{-1} x^{*}) -
              \sigma^{2}/4)/r,
        &x\geq x^{*} \\
\frac{\psi(x)}{\sqrt{x^{*}}\psi'(x^{*})}\;, \phantom{\int^{I^i}}
        &x < x^{*}. \end{cases}
\end{equation}
Especially, the value is a solution of the variational inequality
\[
\min\{((\rho - A)V)(x), V'(x) - x^{-1/2}\} = 0.
\]
\end{example}
\bigskip

We summarize this as follows:

\begin{theorem}
{\bf a)}\
Assume that
\begin{equation}
\mu\leq 2\rho+\sigma^2/4\;.
\end{equation}
Then the value function $V(x)$ of problem (3.29) is
\begin{equation}
V(x)=2\sqrt{x}\;.
\end{equation}
This value is obtained by immediate chattering down to 0.
\bigskip

\noindent
{\bf  b)}\
Assume that
\begin{equation}
\mu> 2\rho+\sigma^2/4\;.
\end{equation}
Then $V(x)$ is given by (3.35). The corresponding optimal policy is
immediate chattering from $x$ down to $x^\ast$ if $x>x^\ast$, and local
time at $x^\ast$ of the downward reflected process $\bar{X}(t)$ at
$x^\ast$ if $x<x^\ast$, where $x^\ast$ is given by (3.34).
\end{theorem}

\section{Discussion on a Special Case}

Our verification Theorem \ref{thm1} covers a large class of state dependent singular stochastic control problems arising in the literature on the rational management of renewable resources.
It is worth emphasizing that there is an interesting subclass (including the case of Example 3.1) of problems where we can utilize our results in order to provide both a lower as well as an upper boundary for the maximal attainable expected cumulative harvesting yield. In order to shortly describe this case,  assume that the underlying dynamics are time homogeneous and independent of each other and, accordingly, that the drift coefficient satisfies $b_i(t,x)=b(x_i)$ and that the volatility coefficient, in turn, satisfies $\sigma_i(t,x)=\sigma_i(x_i)$. Assume also that the price $\pi_i(t,x)=\pi_i(x_i)$ per unit of harvested stock $x_i\in \mathbb{R}_+$ is nonnegative, nonincreasing, and continuously differentiable as a function of the prevailing stock. Given these assumptions, define the nondecreasing and concave function
$$
\Pi_i(x_i) = \int_0^{x_i}\pi_i(v)dv\geq \pi_i(x_i)x_i.
$$
It is now a straightforward example in basic analysis to show by relying on a chattering policy described in our Example 3.1. that in the present case we have 
\begin{align}
J^{(\tilde{\gamma}(m,0))}(0,x)
= \sum_{i=1}^{n}\Pi_i(x_i)\nonumber.
\end{align}
Consequently, under the assumed time homogeneity we observe that the maximal attainable expected cumulative harvesting yield satisfies the inequality
\begin{equation}\label{ineq1}
\sup_\gamma J^{(\gamma)}(0,x)\geq \sum_{i=1}^{n}\Pi_i(x_i)\;.
\end{equation}
On the other hand, utilizing the generalized It{\^o}-Döblin-formula to the mapping $\Pi_i$, invoking the nonnegativity of the value $\Pi_i$,
and reordering terms
yields
\begin{eqnarray}
\Pi_{i}(x_i) &\geq&  - E_{x}\int_{0}^{T_N^{*}}
e^{-\rho s}(\mathcal{G}_\rho^i\Pi_{i})(X_i(s))ds
+ E_{x}\int_{0}^{T_N^{*}} e^{-\rho s}\pi_{i}(X_i(s))d\gamma_i(s) \nonumber\\
&-& E_{x}\sum_{0 \leq s \leq T_N^{*}} e^{-\rho s}[\Pi_{i}(X_i(s)) - \Pi_{i}(X_i(s-)) -
\pi_{i}'(X_i(s-))\Delta X_i(s)]\nonumber,
\end{eqnarray}
where $T^{*}_N$ is an increasing sequence of almost surely finite stopping times converging to $T$ and
$$
(\mathcal{G}_\rho^i\Pi_{i})(x) = \frac{1}{2}\sigma_i^2(x)\pi_i'(x)+b_i(x)\pi_i(x)-\rho\Pi_i(x).
$$
The concavity of the mapping $\Pi_i$ then implies that
$$
\Pi_i(X_i(s)) \leq \Pi_i(X_i(s-)) +\pi_i(X(s-))(X_i(s) - X_i(s-)) = \Pi_i(X_i(s-)) - \pi_i(X_i(s-))\Delta X_i(s).
$$
Hence, we find that for any admissible harvesting strategy $\gamma_i$  we have
\begin{eqnarray*}
E_{x}\int_{0}^{T_N^{*}} e^{-\rho s}\pi_{i}(X_i(s))d\gamma_i(s) \leq  \Pi_{i}(x_i)+ E_{x}\int_{0}^{T_N^{*}}
e^{-\rho s}(\mathcal{G}_\rho^i\Pi_{i})(X_i(s))ds.
\end{eqnarray*}
Summing up the individual values then finally yields
\begin{eqnarray*}
\sum_{i=1}^{n}E_{x}\int_{0}^{T_N^{*}} e^{-\rho s}\pi_{i}(X_i(s))d\gamma_i(s) \leq  \sum_{i=1}^{n}\Pi_{i}(x_i)+ E_{x}\int_{0}^{T_N^{*}}
e^{-\rho s}\sum_{i=1}^{n}(\mathcal{G}_\rho^i\Pi_{i})(X_i(s))ds.
\end{eqnarray*}
Letting $N\uparrow \infty$ and invoking monotone convergence then shows that in the present setting
\begin{eqnarray}
\sup_\gamma J^{(\gamma)}(0,x) &\leq&  \sum_{i=1}^{n}\Pi_{i}(x_i)+ \sup_\gamma E_{x}\int_{0}^{T}
e^{-\rho s}\sum_{i=1}^{n}(\mathcal{G}_\rho^i\Pi_{i})(X_i(s))ds.
\end{eqnarray}
Consequently, in the time homogeneous and independent setting the value which can be attained by a chattering policy can be utilized for the derivation of
both a lower as well as an upper boundary for the value of the optimal harvesting policy. Moreover, in case the generators $(\mathcal{G}_\rho^i\Pi_{i})(X_i(s))$ are bounded above by $M_i$ we observe that
\begin{eqnarray}
\sup_\gamma J^{(\gamma)}(0,x) &\leq&  \sum_{i=1}^{n}\Pi_{i}(x_i)+ \sum_{i=1}^{n}\frac{M_i}{\rho}\left(1-E_{x}[e^{-\rho T}]\right).
\end{eqnarray}

For example, if the underlying evolves as in our 2-dimensional BM example 3.1, we observe that $$
(\mathcal{G}_\rho^i\Pi_i)(x)=x^{-3/2}\theta_i\left(\mu_i x - \frac{\sigma_i^2}{4}-2\rho x^2\right).
$$
Hence, $(\mathcal{G}_\rho^i\Pi_i)(x) \leq (\mathcal{G}_\rho^i\Pi_i)(\tilde{x}_i)$, where
$$
\tilde{x}_i=-\frac{\mu_i}{4\rho}+\frac{1}{4\rho}\sqrt{\mu_i^2+6\sigma_i^2\rho}.
$$
Consequently, we have that
\begin{equation*}
\sup_\gamma J^{(\gamma)}(s,x)\leq 2 e^{-\rho s}\big[\theta_1 \sqrt{x_1}+\theta_2 \sqrt{x_2}\big]+e^{-\rho s}\left((\mathcal{G}_\rho^1\Pi_1)(\tilde{x}_1)+(\mathcal{G}_\rho^2\Pi_2)(\tilde{x}_2)\right)(1-E\left[e^{-\rho T}\right])\;.
\end{equation*}


\begin{thebibliography}{state}

\bibitem[A1]{A1}
Alvarez, L.H.R.
{\em Optimal harvesting under stochastic fluctuations and critical
depensation}, 1998, {\em Mathematical Biosciences}, vol. 152, 63--85.

\bibitem[A2]{A2}
Alvarez, L.H.R. {\em Singular stochastic control in the presence
of a state-dependent yield structure}, 2000, {\em Stochastic
Processes and their Applications}, vol. 86, 323--343

\bibitem[A3]{A3}
Alvarez, L.H.R.
{\em On the option interpretation of rational harvesting planning}, 2000,
{\em Journal of Mathematical Biology}, vol. 40, 383--405.

\bibitem[AS]{AS}
Alvarez, L.H.R. and Shepp, L.A.
{\em Optimal harvesting of stochastically fluctuating populations}, 1998,
{\em Journal of Mathematical Biology}, vol 37, 155--177.

\bibitem[JS]{JS}
Jeanblanc-Picqu\'e, M. and Shiryaev, A.
{\em Optimization of the flow of dividends}, 1995, Russian Math.
Surveys, Vol.~50, 257--277

\bibitem[LES1]{LES1}
Lande, R. and Engen S. and S{\ae}ther B.-E.
{\em Optimal harvesting,
economic discounting and extinction risk in fluctuating populations},
1994, {\em Nature}, vol 372, 88--90.

\bibitem[LES2]{LES2}
Lande, R. and Engen S. and S{\ae}ther B.-E. {\em Optimal harvesting of
fluctuating populations with a risk of extinction}, {\em The American
Naturalist}, 1995, vol 145, 728--745.

\bibitem[L\O{}1]{LO1}
Lungu, E. M. and \O{}ksendal, B.
{\it Optimal harvesting from a population in a stochastic crowded
environment}, 1996, {\em Mathematical Biosciences}, vol. 145, 47--75.

\bibitem[L\O{}2]{LO2}
Lungu, E. M. and \O{}ksendal, B.
{\it Optimal harvesting from interacting populations in a stochastic
environment}, 2001, {\em BERNOULLI}, vol. 7, 527--539.

\bibitem[P]{P}
Protter, P.
{\em Stochastic Integration and Differential Equations},
2004, Second Edition, Springer-Verlag.
\end{thebibliography}
\end{document}